\documentclass[11pt, doublespacing, amsthm]{elsart}
\usepackage{latexsym}
\usepackage{amsfonts}
\usepackage{amsmath}
\usepackage{amsthm}
\usepackage{psfrag}
\usepackage{epsfig}
\usepackage{graphicx}

\DeclareSymbolFont{AMSb}{U}{msb}{m}{n}
\DeclareMathSymbol{\Z}{\mathbin}{AMSb}{"5A}
\def\d{\hbox{-}}

\def\l{,\ldots,}
\newcommand\bbox{\hfill \quad $\Box$ \medbreak}

\newtheorem{lemma}[thm]{Lemma}
\newtheorem{step}{Step}

\begin{document}
\begin{frontmatter}
\title{Counting Paths in Digraphs}
\author{Paul Seymour\thanksref{pds}}
\thanks[pds]{Supported by ONR grant N00014-04-1-0062, and NSF grant DMS03-54465.}
\ead{pds@math.princeton.edu}
\author{Blair D. Sullivan\corauthref{cor}\thanksref{dhs}}
\thanks[dhs]{Partially supported by a Department of Homeland Security (DHS) Dissertation Grant.}
\corauth[cor]{Corresponding author.}
\address{Princeton University, Princeton, NJ 08544}
\ead{bdowling@math.princeton.edu}
\date{January 7, 2007; revised \today}
\end{frontmatter}

\begin{abstract}
Say a digraph is $k$-free if it has no directed cycles of length at most $k$, 
for $k \in \Z^+$. 
Thomass\'{e} conjectured that the number of induced $3$-vertex directed paths in 
a simple $2$-free digraph on $n$ vertices is at most $(n-1)n(n+1)/15$. 
We present an unpublished result of
Bondy proving there are at most $2n^3/25$ such paths, and prove that for the 
class of circular interval digraphs, an upper bound of $n^3/16$ holds. 
We also study the problem of bounding the number of (non-induced) 
$4$-vertex paths in $3$-free digraphs. We show an upper bound of $4n^4/75$ 
using Bondy's result for Thomass\'{e}'s conjecture.
\end{abstract}

\section{Introduction}
We begin with some terminology. All digraphs in this paper are finite.
For a digraph $G$, we denote its vertex and edge sets by $V(G)$ and $E(G)$, respectively. 
Unless otherwise stated, we assume $|V(G)| = n$. 
The members of $E(G)$ are ordered pairs of vertices. We use the notation $uv$ to denote an ordered 
pair of vertices $(u,v)$ (whether or not $u$ and $v$ are adjacent). We only consider 
digraphs which have no loop edges $uu$, and at most one 
directed edge $uv$ for all pairs of vertices $u \neq v$ (are {\em simple}). 
A {\it non-edge} in $G$ is an unordered pair of distinct vertices $u,v$ so that 
$uv,vu$ are both not in $E(G)$. We say a simple digraph $G$ is a {\it tournament} if 
for all pairs of vertices $u \neq v$, exactly one of $uv, vu$ is an edge. 

Given a vertex $v \in V(G)$, we define the set of {\em out-neighbors} to be 
$N^+(v) = \{u: vu \in E(G)\}$ and analogously $N^-(v) = \{u : uv \in E(G)\}$ to be the set of {\em in-neighbors}. 
Let $\delta^+(v) = |N^+(v)|$ and $\delta^-(v) = |N^-(v)|$ denote the {\em out-degree} and {\em in-degree}, respectively.

A {\em directed cycle of length $t$} is a digraph whose vertices and edges can be 
ordered as $v_1, e_1$, $v_2, \dots, e_{t-1}$, $v_t, e_t$ 
with $v_1, \dots, v_t$ distinct vertices, $e_i$ the directed edge $v_iv_{i+1}$ for $i = 1, \dots, t-1$, 
and $e_t = v_tv_1$. We may denote such a cycle as $v_1 \d v_2 \d \cdots \d v_t \d v_1$. 
For an integer $k\ge 0$, let us say a digraph $G$ is {\em $k$-free} 
if there is no directed cycle of $G$ with length at most $k$.  
A digraph is {\em acyclic} if it has no directed cycle. 

A {\em directed walk} in a digraph is a sequence $v_1, e_1, v_2, \dots, e_{t-1}, v_t$ where
$v_1, \ldots, v_t$ are vertices, and $e_i = v_iv_{i+1}$ is an edge for $i = 1, \dots, t-1$; its
{\it length} is $t-1$. 
A {\em directed path} in a digraph is a directed walk where $v_1, \dots, v_t$ are distinct vertices
(its length is $t-1$). We may denote a directed walk (or path) 
as $v_1 \d v_2 \d \cdots \d v_t$. We say a directed path is {\em induced} if every edge 
$v_iv_j$ satisfies $j = i+1$ for $0 \leq i,j \leq t$.  We say a digraph $G$ {\em is 
a directed path} if its vertex set can be labeled $v_1,\dots, v_n$ 
and its edges $e_1, \dots, e_{n-1}$ so that $v_1, e_1, v_2, \dots, e_{n-1}, v_n$ 
is an induced directed path in $G$. Let $W_s(G)$ be the number of distinct directed $s$-vertex walks in a digraph $G$,
$P_s(G)$ the number of distinct $s$-vertex directed paths, and $\tilde{P}_s(G)$ the number 
of distinct induced $s$-vertex directed paths.

The first result of this paper concerns a conjecture of Thomass\'e that the number of
induced $3$-vertex directed paths in a $2$-free digraph on $n$ vertices is at most $(n-1)n(n+1)/15$. 
The best known approximate result is due to Bondy, and is presented in Section \ref{bondysection}. 
We thank him for allowing us to include his proof in this paper. In this paper, we prove a strengthening of Thomass\'e's conjecture for 
``circular interval digraphs''. 

A digraph $G$ is a {\em circular interval digraph} if its vertices can be
arranged in a circle such that for every triple $u,v,w$ of distinct vertices, if $u,v,w$ are in clockwise order
and $uw\in E(G)$, then $uv,vw\in E(G)$. This is equivalent to saying that the vertex set of $G$ can be 
numbered as $v_1\l v_n$ such that for $1\le i\le n$, the set of out-neighbors of $v_i$ is 
$\{v_{i+1}\l v_{i+a}\}$ for some $a\ge 0$, and the set of in-neighbors of $v_i$ is 
$\{v_{i-b}\l v_{i-1}\}$ for some $b\ge 0$, reading subscripts modulo $n$. 

In Section \ref{CIGsection}, we show: 
\begin{thm} If $G$ is a $2$-free circular interval digraph on $n$ vertices, then $\tilde{P}_3(G) \leq n^3/16$. 
\end{thm}

The second result of this paper was motivated by the following problem. 
For integer $t$, let $\alpha_t$ be the minimum constant so that 
all $n$-vertex digraphs with minimum out-degree at least $\alpha_t n$ have a directed cycle of length at most $t$ 
(it can be proved that $\alpha_t$ exists). 
The Caccetta-H\"{a}ggkvist conjecture \cite{CH} is that $\alpha_t = 1/t$. 
A number of papers have focused on the special case of getting an upper bound on $\alpha_3$ that 
is as close to $1/3$ as possible. The most recent result by Shen \cite{shen:hhk} slightly 
tightens an argument of Hamburger, Haxell, and Kostochka \cite{hhk} and proves $\alpha_3 \leq .3530381$. 

One possible approach for finding upper bounds on $\alpha_3$ is to find bounds on the 
number of short directed walks in $3$-free digraphs. If $G$ is a digraph on $n$ vertices with minimum out-degree $d$, 
then $W_s(G) \geq d^{s-1} n$, and hence a bound of the form $W_s(G) \leq (c_{s}n)^{s}$ for $3$-free digraphs 
$G$ would prove there is a vertex of out-degree at most $(c_{s})^{\frac{s}{s-1}}n$. 
 
We observe that if $G$ is $3$-free, then $W_4(G) = P_4(G)$. We will show: 
\begin{thm}\label{ip4thm} If $G$ is a $3$-free digraph on $n$ vertices, then $P_4(G) \leq \frac{4}{75}n^4.$ 
\end{thm}

Note that there exists an infinite family of $3$-free graphs where 
$P_4(G)/n^4 \rightarrow \frac{25}{512} \approx .0488$ as 
$n \rightarrow \infty$. These graphs are given by taking four acyclic tournaments $S_1, \dots, S_4$, each on
$n/4$ vertices and adding the edges $uv$ where $u \in S_i$ and $v \in S_{i+1}$ for $i = 1,2,3$, 
as well as those from $S_4$ to $S_1$. This shows that using an upper bound on $c_4$ to 
imply a bound on $\alpha_3$ will not lead to an improvement of Shen's result. Theorem \ref{ip4thm} implies that any 
$3$-free digraph on $n$ vertices has minimum out-degree at most $\sqrt[3]{4/75}n \approx .3764n$. 

\section{Thomass\'e's Conjecture and Bondy's Result} \label{bondysection}

There was a workshop on the Caccetta-H\"aggkvist Conjecture at the American Institute of Mathematics (AIM) in January of 2005. 
In discussions at that workshop, Thomass\'e proposed the following conjecture, and Bondy proved a partial result that we use in Section
\ref{p4section}.

\begin{conj}[Thomass\'{e}]\label{thomasseconj}
If $G$ is a $2$-free digraph on $n$ vertices, then 
$$\tilde{P}_3(G) \leq \frac{(n-1)n(n+1)}{15}.$$
\end{conj}

This is tight on the following infinite family of digraphs: Let $G_0$ be the digraph 
consisting of a single vertex and no edges. Define $G_i$ for $i \geq 1$ to be the digraph obtained
by taking four disjoint copies of $G_{i-1}$ (call them $D_1, D_2, D_3, D_4$) and forming the 
digraph with vertex set $V(G_i) = \bigcup_{j=1}^{4} V(D_j)$ and edge set 
\begin{displaymath}
E(G_i) = \left(\bigcup_{j=1}^4 E(D_j)\right) \cup \{uv : u \in D_j, v \in D_{j+1}, j = 1,2,3,4\},
\end{displaymath} 
where $D_5$ means $D_1$. In other words, arrange four copies of $G_{i-1}$ in a square 
and put in all edges between consecutive copies in a clockwise direction. 
It is easy to check inductively that $\tilde{P}_3(G_i) = (n_i-1)n_i(n_i+1)/15$, where $n_i = 4^i = |V(G_i)|$.

The best known result for general $2$-free digraphs is due to Bondy, whom we thank for permission to include his result here.

\begin{thm}[Bondy] \label{bondythm}
If $G$ is a $2$-free digraph on $n$ vertices, then $\tilde{P}_3(G) \leq \frac{2}{25}n^3$. 
\end{thm}

\begin{figure}[tbp]
\begin{center}
\includegraphics[width=.6\textwidth]{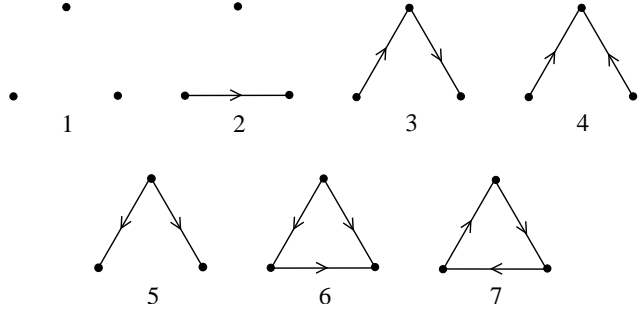}
\caption{The $3$-vertex digraphs.}
\label{3vd}
\end{center}
\end{figure}

\begin{proof}
There are seven digraphs on three vertices up to isomorphism, which we call types $1, \dots, 7$ as shown in Figure~\ref{3vd}.
Given a digraph $G$ with vertex set $\{v_i:1 \leq i \leq n\}$, let
$d_i^-$ and $d_i^+$ denote the in-degree and out-degree of $v_i$ ($1 \leq i
\leq n$) and $s_j$ the number of induced subgraphs of type $j$ in $G$ ($1 \leq j \leq 7$). 

The following five equations hold:
\begin{eqnarray*}
s_{1}+s_{2}+s_{3}+s_{4}+s_{5}+s_{6}+s_7&=&\binom{n}{3}\\
s_{2}+2s_{3}+2s_{4}+2s_{5}+3s_{6}+3s_7&=& \frac{1}{2}(n-2)\sum_i(d_i^-+d_i^+)\\
s_{3}+s_{6}&=&\sum_i\binom{d_i^-}{2}\\
s_{4}+s_{6}+3s_7&=&\sum_i d_i^-d_i^+\\
s_{5}+s_{6}&=&\sum_i\binom{d_i^+}{2}.
\end{eqnarray*}
We prove an upper bound on $s_4 = \tilde{P}_3(G)$ as follows:
\begin{eqnarray*}
s_4 &\leq &    \frac{2}{5}s_2 + \frac{1}{10}s_3 +s_4+\frac{1}{10}s_5+\frac{9}{5}s_7\\
    &=&    \frac{2}{5}(s_{2}+2(s_{3}+s_{4}+s_{5})+3(s_{6}+s_7))-\frac{7}{10}(s_3 + s_5+ 2s_6)
            +\frac{1}{5}(s_4+s_6+3s_7)\\
    &=&     \frac{n-2}{5}\sum_i(d_i^-+d_i^+) - \frac{7}{20}\sum_i(((d_i^-)^2 - d_i^-) + (d_i^+)^2 - d_i^+)) 
            +\frac{1}{5}\sum_id_i^-d_i^+\\
    &=&     \frac{n}{5}\sum_i(d_i^-+d_i^+) - \frac{7}{20}\sum_i((d_i^-)^2  + (d_i^+)^2) 
            +\frac{1}{5}\sum_id_i^-d_i^+ -\frac{1}{20}\sum_i (d_i^+ + d_i^-)\\
    &=&     \frac{2n^3}{25} -\frac{1}{10}\sum_i(d_i^--d_i^+)^2 - 
	    \frac{1}{4}\sum_i\left(\frac{2n}{5}-d_i^-\right)^2 -\frac{1}{4}\sum_i\left(\frac{2n}{5}-d_i^+\right)^2\\
   &\leq&   \frac{2n^3}{25}, 
\end{eqnarray*}
which proves Theorem \ref{bondythm}.
\end{proof}

\section{Induced $3$-vertex Paths in Circular Interval Digraphs}\label{CIGsection}

The main result of this section is: 
\begin{thm}\label{CIGthm}
If $G$ is a $2$-free circular interval digraph on $n$ vertices, then $\tilde{P}_3(G) \leq n^3/16.$ 
\end{thm} 

We first show this is best possible. Let $G$ be a $2$-free circular interval digraph. For $u,v \in V(G)$ let 
\begin{displaymath}
d(u,v) = \begin{cases}
1+ |\{w\in V(G): u,w,v\text{ distinct, in clockwise order}\}| & \text{if } u \neq v \\
0 & \text{if } u = v.
\end{cases}
\end{displaymath}
For every pair $uv$, we say its {\em length} is $d(u,v)$. For integer $\beta$, 
let $G_{\beta}$ be the circular interval digraph on $n$ vertices with 
$E(G_\beta) = \{uv\, :\, 0 < d(u,v) \leq \beta\}$. 

\begin{lemma} For infinitely many values of $n$, there are circular interval digraphs on 
$n$ vertices with exactly $n^3/16$ induced $3$-vertex paths.
\end{lemma}

\begin{proof}
Let $n$ be chosen so that $\beta = (3n-4)/8$ is an integer. 
A straightforward computation shows the number of induced $3$-vertex paths in $G_\beta$ 
is $n(n-2\beta-1)(2\beta - n/2 + 1)$. Then $G_{(3n-4)/8}$ has $(n-(3n-4)/4 -1)((3n-4)/4 -n/2 + 1) = n^3/16$
induced $3$-vertex paths.
\end{proof}

To prove Theorem \ref{CIGthm}, we first need a few definitions and lemmas. 
Given $X \subseteq V(G)$, define $G|X$ to be the digraph with vertex set $X$ and 
edge set $\{uv \in E(G): u,v \in X\}$. For $Y \subseteq E(G)$, we write 
$G\setminus Y$ for the digraph with vertex set
$V(G)$ and edge set $E(G)\setminus Y$. If $Y = \{e\}$ where $e = uv$, then 
we may abbreviate as $G\setminus Y = G\setminus e  = G\setminus uv$. 
If $Z$ is a set of non-edges of $G$, we write $G + Z$ for the digraph with vertex set $V(G)$
and edge set $E(G) \cup \{uv:\, uv \in Z\}$. Analogously, if $Z = \{f\}$ with $f = uv$, we may 
write $G + f = G + uv = G + Z$.
 
We define $\alpha_G$ to be the length of a shortest non-edge in $G$ (if $G$ has a non-edge, and otherwise
we let $\alpha_G = \infty$). We also define $\beta_G$ to be the length of a longest edge in $G$ (if $G$ has an edge, 
and otherwise we let $\beta_G = 0$). 

\begin{lemma}\label{stayCIGobs} Let $G = (V,E)$ be a $2$-free circular interval digraph. 
Let $X$ be a set of longest edges in $G$ and $Y$ a set of shortest non-edges in $G$ so that for all $u,v \in V(G)$, 
$uv$ and $vu$ are not both in $Y$. Then $G\setminus X$ and $G + Y$ are 
$2$-free circular interval digraphs. Additionally, if $\alpha_G \leq \beta_G$, then
the digraph $(G\setminus X) + Y$ is also a $2$-free circular interval digraph. 
\end{lemma}

Let $\xi(G)$ denote the number of pairs $(uv,wx)$ where $uv$ is an edge of $G$, 
$wx$ is a non-edge and $d(u,v) > d(w,x)$ ($u,v$ are not necessarily distinct from $w,x$). 
For a fixed $n \geq 4$, say a digraph $G$ is {\em optimal} if among all $2$-free circular 
interval digraphs on $n$ vertices, it has the maximum number of $3$-vertex induced directed paths and subject 
to this, $\xi(G)$ is minimum.

We now show optimal digraphs do not have edges with length at least $n/2$.

\begin{lemma}\label{n/2lemma} If $G$ is an optimal digraph on $n$ vertices, then $\beta_G < n/2$.
\end{lemma}
\begin{proof}
If $G$ has no edges, then $\beta_G = 0 < n/2$, so we may assume $E(G) \neq \emptyset$. Suppose $\beta =
\beta_G \geq n/2$ and let $e = uv$ be an edge of length $\beta$. Let $G' = G \setminus e$, 
which is also a $2$-free circular interval digraph by Lemma \ref{stayCIGobs}. Define $c = |N^+(v) \cap N^-(u)|$. Then
\begin{equation}\label{n/2_eqnA}
\tilde{P}_3(G') = \tilde{P}_3(G) - (\delta^-(u) + \delta^+(v) - 2c) + (\beta -1 + c).
\end{equation} 
Since $N^+(v), N^-(u)\setminus N^+(v), \{u,v\}$, and $\{w : u,w,v$ are in clockwise order, $w \neq u,v\}$ 
are disjoint sets in $V(G)$, we have 
\begin{equation}\label{n/2_eqnB}
\delta^+(v) + (\delta^-(u) -c)+ 2 + (\beta -1) \leq n.
\end{equation} 
Rearranging (\ref{n/2_eqnB}) gives $\delta^-(u) + \delta^+(v) \leq n-1-\beta + c$, and 
substituting for $\delta^-(u) + \delta^+(v)$ in  (\ref{n/2_eqnA})
gives $\tilde{P}_3(G') \geq \tilde{P}_3(G) - (n-1-\beta + c - 2c) + (\beta -1 + c)$ or 
\begin{equation*}
\tilde{P}_3(G') \geq \tilde{P}_3(G) - (n- 2\beta - 2c). 
\end{equation*}
Since $\tilde{P}_3(G') \leq \tilde{P}_3(G)$ because $G$ is optimal, we have $n-2\beta-2c \geq 0$. 
Since $\beta \geq n/2$, it follows that $\beta = n/2$, $c = 0$, and $\tilde{P}_3(G') = \tilde{P}_3(G)$. Hence $d(v,u) = 
n-d(u,v) = n/2 \geq 2$ ($n \geq 4$ since $G$ is optimal). Then there is at least one vertex $w$ so that $u,v,w$ appear
in clockwise order. Since $c = 0$, one of $vw,wu$ must be a non-edge, and thus $\alpha_G < n/2$. But then 
$\xi(G') < \xi(G)$, contradicting the optimality of $G$.
This proves Lemma \ref{n/2lemma}.
\end{proof}

We now prove a straightforward lemma giving upper and lower bounds on the vertex degrees in an optimal digraph.

\begin{lemma}\label{outdegbounds} For every vertex $v$ in an optimal digraph $G$, 
\begin{displaymath}
\alpha_G -1 \leq \delta^+(v), \delta^-(v) \leq \beta_G.
\end{displaymath}
\end{lemma}
\begin{proof} Let $v$ be a vertex in an optimal digraph $G$ on $n$ vertices. 
Certainly $\alpha_G$ is finite, as otherwise $G$ has no induced directed paths
of length greater than one. 
Suppose $\delta^+(v) < \alpha_G -1$. We first show that $v$ has a non-neighbor.
We have $\alpha_G \le \beta_G + 1$, 
since if $v_iv_j$ is a shortest non-edge then $v_iv_{j-1}$ is an edge, and 
therefore has length at most $\beta_G$.
Then since $\beta_G <n/2$ by Lemma \ref{n/2lemma}, 
it follows that $\alpha_G < n/2 + 1$. Now $\delta^+(v) \leq \alpha_G -2$ implies  
$\delta^+(v) < n/2 -1$.  Since $v$ has in-degree at most
$(n-1)/2$ by Lemma \ref{n/2lemma}, it has less than $(n/2-1) + (n-1)/2 = n -3/2$ neighbors. 
Consequently $v$ has a non-neighbor, and we let $u$ be the first vertex following $v$ in the 
clockwise order for which $vu$ is a non-edge. Then $vu$ has length $\delta^+(v) + 1 < \alpha_G$, 
a contradiction to the definition of $\alpha_G$. Analogously, $\delta^-(v) \geq \alpha_G -1$. 
Now, suppose $\delta^+(v) > \beta_G$. Then the 
edge from $v$ to its last clockwise out-neighbor has length $1+ (\delta^+(v)-1) > \beta_G$. This
contradicts the definition of $\beta_G$. Again, $\delta^-(v) \leq \beta_G$ by an analogous argument.
\end{proof}

This allows us to give a lower bound on $\alpha_G$ in optimal digraphs $G$.

\begin{lemma}\label{n/4lemma} If $G$ is an optimal digraph on $n$ vertices, then $\alpha_G > n/4$.
\end{lemma}
\begin{proof} If $G$ has no non-edges, then $\alpha_G = \infty > n/4$. We may assume $G$ has a non-edge. 
Suppose $\alpha = \alpha_G \leq n/4$ and let $e = uv$ be a non-edge of length $\alpha$. 
Let $G' = G + e$, which is also a $2$-free circular interval digraph by 
Lemma \ref{stayCIGobs}. Define $c = |N^+(v) \cap N^-(u)|$. Then 
\begin{equation*}
\tilde{P}_3(G') = \tilde{P}_3(G) - (\alpha-1 + c) + (\delta^-(u) + \delta^+(v) - 2c).
\end{equation*}
Since $\tilde{P}_3(G') \leq \tilde{P}_3(G)$ because $G$ is optimal, we have 
\begin{equation}\label{n/4_eqnB} 
\alpha-1 + 3c \geq \delta^-(u) + \delta^+(v).
\end{equation}
Suppose $c = 0$. Then since $\delta^+(v), \delta^-(u) \geq \alpha -1$ by Lemma \ref{outdegbounds}, we have $\alpha = 1$ and $\delta^-(u) = \delta^+(v) = 0$. Letting $w$ be the vertex immediately following $v$ 
in the circular order, we see that $G + \{uv,vw\}$ has more induced $2$-edge paths than $G$, 
contradicting its optimality. Thus $c > 0$. Now $N^+(v), N^-(u)\setminus N^+(v)$, 
$\{w: u,w,v $ are in clockwise order, $w \neq u,v\}$, and $\{u,v\}$ form a 
partition of $V(G)$, so 
\begin{equation}\label{n/4_eqnC}
\delta^+(v) + \delta^-(u) - c + (\alpha -1) + 2 = n.
\end{equation}
We observe that Lemmas \ref{n/2lemma} and \ref{outdegbounds} imply  
\begin{equation}\label{n/4_eqnD}
\delta^+(v) + \delta^-(u) \leq 2\beta \leq n-1.
\end{equation}
Taking the combination (\ref{n/4_eqnB})$ + $(\ref{n/4_eqnC})$- 2\cdot$(\ref{n/4_eqnD}) and 
simplifying gives
$4\alpha \geq n$, so $\alpha = n/4$ and we have equality in both (\ref{n/4_eqnB}) and 
(\ref{n/4_eqnD}). The equality in (\ref{n/4_eqnD}) implies $\beta_G \geq (n-1)/2$, 
so $\beta_G > \alpha = n/4$. It follows that $\xi(G') < \xi(G)$. Yet the equality in (\ref{n/4_eqnB})
tells us $\tilde{P}_3(G) = \tilde{P}_3(G')$, contradicting the optimality of $G$.
This proves Lemma \ref{n/4lemma}.
\end{proof}

\begin{lemma}\label{nonemptylemma} If $G$ is an optimal digraph and 
$uv$ is a shortest non-edge in $G$, then $N^+(v) \cap N^-(u) \neq \emptyset$.
\end{lemma}
\begin{proof} Suppose not, and let $uv$ be a non-edge of length $\alpha_G$ in an optimal digraph $G$ on $n$ vertices. 
Then by Lemma \ref{n/4lemma} and the fact that $n \geq 4$, $\alpha_G > n/4 \geq 1$.
Let $G' = G + uv$ and $\alpha = \alpha_G$. Then 
\begin{equation*}
\tilde{P}_3(G') = \tilde{P}_3(G) + \delta^+(v) + \delta^-(u) - (\alpha-1).
\end{equation*}
Since $\tilde{P}_3(G') \leq \tilde{P}_3(G)$ by optimality, $\delta^+(v) + \delta^-(u) \leq \alpha-1$. But 
by Lemma \ref{outdegbounds}, $\delta^+(v), \delta^-(u) \geq \alpha -1$. Then $2\alpha-2 \leq \alpha -1$, or
$\alpha \leq 1$, a contradiction. 
This proves Lemma \ref{nonemptylemma}.
\end{proof}

We can now prove that in an optimal digraph $G$, $\alpha_G + \beta_G$ is approximately $3|V(G)|/4$. Let
\begin{equation*}
\epsilon_G = 
\begin{cases} 
0 & \beta_G > \alpha_G \\
1 & \beta_G \leq \alpha_G. 
\end{cases}
\end{equation*}

\begin{lemma}\label{alphabetalemma} If $G$ is an optimal digraph on $n$ vertices, then 
\begin{displaymath}
\frac{3n}{4} - \frac{1}{2} - \frac{\epsilon_G}{4} < \alpha_G + \beta_G < \frac{3n}{4} + \frac{1}{2} + \frac{\epsilon_G}{4}.
\end{displaymath}
Additionally, if some vertex is incident with a longest edge, but with no shortest non-edge, 
then $\alpha_G + \beta_G < 3n/4 + \epsilon_G /4$, and if some vertex is incident with a shortest
non-edge but with no longest edge, then $\alpha_G + \beta_G > 3n/4 - \epsilon_G /4$.
\end{lemma}

\begin{proof}
Let $G$ be an optimal digraph on $n$ vertices with 
$\alpha = \alpha_G$, $\beta = \beta_G$, and $\epsilon = \epsilon_G$. 
\setcounter{step}{0}
\begin{step}\label{abupper} $\alpha + \beta < 3n/4 + 1/2 + \epsilon /4$, and if some vertex is 
incident with a longest edge but no shortest non-edge, then $\alpha + \beta < 3n/4 + \epsilon /4$.
\end{step}
If $G$ has no edges, then $\beta = 0$, and since $\alpha \leq \beta + 1$ by Lemma \ref{outdegbounds}, we have 
$\alpha + \beta \leq 1 \leq 3n/4$ (since $n \geq 4$), as required. Thus $E(G) \neq \emptyset$.
Let $uv$ be a longest edge in $G$ and $G' = G \setminus uv$. For notational convenience, 
set $\delta^+(v) = a$, $\delta^-(u) = b$, and $|N^+(v) \cap N^-(u)| = c$. 
The number of induced $3$-vertex paths in $G'$ is
$$\tilde{P}_3(G') =  \tilde{P}_3(G) + (b-c) + (a -c) + (\beta -1 + c).$$
Since $G$ is optimal, $\tilde{P}_3(G') \leq \tilde{P}_3(G)$, and strict inequality holds if $\beta > \alpha$ 
(because then $\xi(G') < \xi(G)$), it follows that:
\begin{equation}\label{lossgain1}
3c < a + b - \beta + 1 + \epsilon.
\end{equation}

Suppose that no vertex is non-adjacent to both $u$ and $v$.  Then 
counting the vertices, we have $(c + \beta -1)+ (a-c) + (b-c) + 2 = n$, or 
$c = a + b + \beta +1 -n$. Substituting for $c$ in (\ref{lossgain1}) gives 
$2(a + b + 1) + 3(\beta -n) < \epsilon - \beta$, or 
\begin{equation*}
4\beta + 2 + 2(a + b) < 3n + \epsilon.  
\end{equation*}
Since $a,b \geq \alpha -1$ by Lemma \ref{outdegbounds}, it follows that 
$4\beta + 2 + (4\alpha - 4) < 3n + \epsilon$, or 
\begin{displaymath}
\alpha + \beta < 3n/4 + 1/2 + \epsilon/4.
\end{displaymath}
Note that if $a \geq \alpha$ or $b \geq \alpha$ (one of the endpoints of $uv$ is not incident with a shortest non-edge),
we have $4\beta + 2 + (4\alpha -2) < 3n + \epsilon$, or
\begin{displaymath}
\alpha + \beta < 3n/4 + \epsilon/4.
\end{displaymath}

Thus we may assume there is a vertex $y$ non-adjacent to $u$ and $v$. In this case, 
$a + b + (\beta -1) + 3 \leq n$, so $2\alpha + \beta  \leq n$. 
We know $\alpha \geq (n+1)/4$ by Lemma \ref{n/4lemma} (since $\alpha \in \Z$), proving 
\begin{displaymath}
\alpha + \beta < 3n/4 - 1/4 < 3n/4 + \epsilon /4.
\end{displaymath}
This proves Step \ref{abupper}.~\bbox

\begin{step}\label{ablower} $\alpha + \beta > 3n/4 -1/2-\epsilon /4$, and if some vertex 
is incident with a shortest non-edge but no longest edge, then $\alpha + \beta > 3n/4 - \epsilon /4$.
\end{step}
If $G$ has no non-edges, then $\alpha = \infty$, yet $\alpha \leq \beta + 1 < n/2 + 1$ by Lemmas \ref{outdegbounds} and
\ref{n/2lemma}, a contradiction. Then let $uv$ be a shortest non-edge in $G$, and $G' = G + uv$. 
For notational convenience, set $\delta^+(v) = a$, $\delta^-(u) = b$, 
and $|N^+(v) \cap N^-(u)| = c$. The number of induced $3$-vertex paths in $G'$ is
$$\tilde{P}_3(G') = \tilde{P}_3(G) + (b-c) + (a-c) - (c + \alpha -1).$$
Since $G$ is optimal, $\tilde{P}_3(G') \leq \tilde{P}_3(G)$, and strict inequality holds if $\beta > \alpha$ 
(because then $\xi(G') < \xi(G)$), it follows that:
\begin{equation}\label{lossgain2}
a + b - \alpha + 1 < 3c + \epsilon.
\end{equation} 
We know $\alpha + 1 + a + b-c = n$ by Lemma \ref{nonemptylemma}. 
We solve for $c = \alpha + a + b + 1 -n$, and substitute into equation (\ref{lossgain2}). 
This gives $a + b - \alpha + 1 -\epsilon < 3(\alpha + a + b + 1 -n)$, or
\begin{equation*}
3n < 4\alpha + 2 (a+b) + 2 + \epsilon.
\end{equation*}
Since $a,b \leq \beta$ by Lemma \ref{outdegbounds}, 
$3n < 4(\alpha + \beta) + 2 + \epsilon$, or 
\begin{displaymath}
\alpha + \beta > 3n/4 - 1/2 - \epsilon/4.
\end{displaymath}
If $a < \beta$ or $b < \beta$ (one of the endpoints of $uv$ is not incident with a longest edge), 
we instead have $3n <  4\alpha + 2\beta + 2(\beta-1) + \epsilon$, or 
\begin{displaymath}
\alpha + \beta > 3n/4 - \epsilon/4, 
\end{displaymath}
as desired. This proves Step \ref{ablower}, and completes the proof of Lemma \ref{alphabetalemma}.
\end{proof}

We define $\gamma_G = 4(\alpha_G + \beta_G)- 3n$, where $|V(G)| = n$.

\begin{lemma}\label{gammalemma} Let $G$ be an optimal digraph on $n$ vertices with $\beta_G > \alpha_G$. Then 
$-1 \leq \gamma_G \leq 1.$ Furthermore, $\gamma_G = -1$ if some vertex is incident 
with a longest edge and with no shortest non-edge, and 
$\gamma_G = 1$ if some vertex is incident with a shortest non-edge and with no longest edge.
\end{lemma}
\begin{proof}
Since $\gamma_G = 4(\alpha_G + \beta_G) -3n$ and $\epsilon_G = 0$, Lemma \ref{alphabetalemma} implies 
\begin{displaymath}
-2 = 4(3n/4 - 1/2) - 3n < \gamma_G < 4(3n/4 + 1/2) - 3n = 2.
\end{displaymath}
Since $\gamma_G$ is an integer, this is equivalent to $-1 \leq \gamma_G \leq 1$. Furthermore, 
if some vertex is incident with a longest edge and no shortest non-edge, Lemma \ref{alphabetalemma}
proves $\gamma_G < 0$. Since $\gamma_G \geq -1$, this implies $\gamma_G = -1$. Similarly, 
if some vertex is incident with a shortest non-edge and no longest edge, 
we have $\gamma_G > 0$, which combined with $\gamma_G \leq 1$ implies $\gamma_G = 1$. This proves Lemma \ref{gammalemma}.
\end{proof}

We now prove several more facts about optimal digraphs. We say a set of vertices 
$X$ in a digraph $G$ is {\it stable} if $uv \notin E(G)$ for all $u,v \in X$; that is, $G|X$ has no edges.  

\begin{lemma}\label{stablesetlemma} 
If $G$ is an optimal digraph, it has no stable set of size at least $3$.
\end{lemma}
\begin{proof}
Suppose not and let $|V(G)| = n$. Take a stable set $\{u,v,w\}$ so that $d(u,v)$ is minimum, and let 
$k = d(u,v)$. Then by Lemma \ref{n/4lemma}, $\alpha_G > n/4 \geq 1$, so every 
vertex has at least one out-neighbor. It follows that $k \geq 2$. 
Let $G' = G + uv$. Then $G'$ is a circular interval digraph, since $d(u,v)$'s  minimality
implies $uw$ and $wv$ are edges for all $w$ between $u$ and $v$ in the circular order. 
Since $|N^+(v)\cap N^-(u)| = 0$, it follows that
\begin{equation*}
\tilde{P}_3(G') = \tilde{P}_3(G) + \delta^+(v) + \delta^-(u) - (k-1).
\end{equation*}
Since $\tilde{P}_3(G') \leq \tilde{P}_3(G)$ by optimality, $\delta^+(v) + \delta^-(u) \leq k-1$.

Let $y$ be the furthest out-neighbor of $v$ in $G$. Then $v$ is non-adjacent to the next vertex
in the circular order (call it $a$), and the non-edge $va$ is part of a stable set of size three, 
namely $\{v,a,u\}$ (if $u$ were adjacent to $a$ then there could not be a vertex 
$w$ to which both $u$ and $v$ were non-adjacent). 
This means the length of $va$ is at least $k$ by choice of $uv$, so $\delta^+(v) \geq k-1$. An
analogous argument shows $d^-(u) \geq k-1$. Since $\delta^+(v) + \delta^-(u) \leq k-1$, 
it follows that $k=1$, a contradiction. This proves Lemma \ref{stablesetlemma}.
\end{proof}

We say a pair of vertices $uv$ in an optimal digraph $G$ is {\em extreme} if $uv$ is a longest edge 
or a shortest non-edge in $G$. 

\begin{lemma}\label{extremelemma} Let $G$ be an optimal digraph with $\beta_G > \alpha_G$, and 
$u,v,w$ vertices appearing in clockwise order. Then not all of $uv,vw,wu$ are extreme pairs. Additionally, 
if some two of them are extreme, then either all three pairs are edges, or two are edges and the third is a
shortest non-edge.
\end{lemma}
\begin{proof} Let $G$ be an optimal digraph on $n$ vertices with $\beta_G > \alpha_G$. 
We begin by proving that no vertex is in two shortest non-edges.
Assume not, and let $v$ be a vertex with $uv$ and $vw$ non-edges of length $\alpha_G$. 
Hence $u,v,w$ appear in clockwise order. Then by Lemma \ref{stablesetlemma}, $u$ and $w$ must be adjacent.
If $uw \in E(G)$, $G$ is not a circular interval graph, a contradiction.  
Thus $wu \in E(G)$, and say it has length $L$. Let $G' = G  + vw$. 
This is a circular interval digraph by Lemma \ref{stayCIGobs}. Then 
\begin{equation*}
\tilde{P}_3(G') = \tilde{P}_3(G) - (\alpha_G - 1 + \delta^+(w) - L) + L + (\alpha_G -1)-(\delta^+(w)-L).
\end{equation*}
Since $\tilde{P}_3(G') \leq \tilde{P}_3(G)$ by optimality, $3L-2\delta^+(w) \leq 0$, or 
\begin{equation*}
L \leq \frac{2\delta^+(w)}{3} \leq \frac{2\beta_G}{3}.
\end{equation*}
Now since $2\alpha_G + L = n$, we have 
\begin{equation}\label{extreme:abnineq}
n \leq 2\alpha_G + \frac{2\beta_G}{3}. 
\end{equation}
We note that since $\beta_G > \alpha_G$ and $v$ is 
incident with two shortest non-edges, $v$ cannot be incident with a longest edge. Then Lemma \ref{gammalemma} gives
$\gamma_G = 1$, or 
\begin{equation}\label{extreme:3n+1/4}
\alpha_G + \beta_G = \frac{3n+1}{4}.
\end{equation}
Combining equations (\ref{extreme:abnineq}) and (\ref{extreme:3n+1/4}), we have 
\begin{displaymath}
n \leq \frac{2(\alpha_G + \beta_G)}{3} + \frac{4\alpha_G}{3} = \frac{3n+1}{6} + \frac{4\alpha_G}{3}. 
\end{displaymath}
This implies
\begin{equation*}
\alpha_G \geq \frac{3n-1}{8}. 
\end{equation*}
Then since $\alpha_G + \beta_G = (3n+1)/4$, $\beta_G \leq (3n+3)/8$. Yet $\beta_G > \alpha_G$, and both are integers. 
This implies there are two integers in the range $[(3n-1)/8, (3n+3)/8]$, a contradiction. This proves no 
vertex is incident with two shortest non-edges.

We now show no triple of vertices forms three longest edges. Let $u,v,w$ be in clockwise order, 
and assume $uv,vw,wu$ are all edges of length $\beta_G$. 
Then $3\beta_G = n$. Since $\beta_G > \alpha_G$, this implies $n > 2\beta_G + \alpha_G$. Now, 
Lemma \ref{alphabetalemma} gives $\alpha_G + \beta_G  > 3n/4 - 1/2$ (since $\epsilon_G = 0$), so 
$n > \beta_G + 3n/4 - 1/2$, or $\beta_G < n/4 + 1/2$. Then $\alpha_G < \beta_G < (n+2)/4$ implies
$\alpha_G < (n+1)/4$, a contradiction to Lemma \ref{n/4lemma}. 

This proves that for $u,v,w$ in clockwise order, not all of $uv,vw,wu$ are extreme pairs. Additionally, 
it proves that if two are extreme pairs, at least one must be a longest edge. To complete the proof of the theorem, 
we need to show that if two of $uv,vw,wv$ are longest edges, the third pair must be an edge and that 
if there is a shortest non-edge among $uv,vw,wv$, the other two pairs must be edges.

We first prove there do not exist $u,v,w \in V(G)$ so that $uv, vw$ are edges of length $\beta_G$ and 
$wu$ is a non-edge. Suppose such $u,v,w$ exist. It follows that $u,v,w$ are in clockwise order. Since
all three pairs are not extreme, $wu$ has length $L > \alpha_G$. Then $2\beta_G + L = n$, or 
$2\beta_G + \alpha_G < n$. Since $v$ is incident with two longest edges, it cannot be incident with a shortest non-edge, 
and Lemma \ref{gammalemma} implies $\gamma_G = -1$, or $\alpha_G + \beta_G = (3n-1)/4$. We then 
have $(3n-1)/4 + \beta_G < n$, or $\beta_G < (n+1)/4$. Since $\alpha_G < \beta_G$, this contradicts Lemma \ref{n/4lemma}.

Finally, suppose there are $u,v,w \in V(G)$ so that $uv,vw$, and $wu$ consist of a shortest non-edge, a longest edge, and
a non-edge of length $L > \alpha_G$. Then $\alpha_G + \beta_G + L = n$ implies $2\alpha_G + \beta_G < n$. 
Since $\alpha_G + \beta_G > (3n-2)/4$ by Lemma \ref{alphabetalemma}, we see that $\alpha_G + (3n-1)/4 < n$, 
or $\alpha_G < (n+1)/4$. Again, this contradicts Lemma \ref{n/4lemma}.

This proves Lemma \ref{extremelemma}.
\end{proof}

Given an optimal digraph $G$, let $S = v_1, f_1$, $v_2, f_2$, $v_3, f_3, \dots, f_k, v_{k+1}$ 
be a sequence where the $v_i$ are vertices of $G$, 
and the $f_i = v_iv_{i+1}$ are extreme pairs of $G$. We say $S$ is an {\em alternating sequence} 
if it satisfies the conditions: 
\begin{enumerate}
\item $f_i \neq f_j$ for $i \neq j$.
\item For $1 \leq i \leq k-1$, if $f_i$ is an edge, then $f_{i+1}$ is a non-edge.
\item For $1 \leq j \leq k-1$, if $f_j$ is a non-edge, then $f_{j+1}$ is an edge.
\end{enumerate}
In other words, $S$ is an alternating sequence of longest edges and shortest non-edges. 
Define $X_S$ to be the set of longest edges in $S$
and $Y_S$ to be the set of shortest non-edges. We say a sequence $S$ is an 
{\it augmenting sequence} if it is a maximal alternating sequence. 

\begin{lemma}\label{Satleast3} Let $G$ be an optimal digraph on $n$ vertices with $\beta_G > \alpha_G$.
If $\alpha_G + \beta_G \leq 3n/4$, then every shortest non-edge has a longest edge 
incident with each of its endpoints. If $\alpha_G + \beta_G \geq 3n/4$, every longest edge
has a shortest non-edge incident with each of its endpoints. Consequently, every 
augmenting sequence in $G$ has at least $3$ extreme pairs.
\end{lemma}
\begin{proof} Let $G$ be an optimal digraph on $n$ vertices with $\beta_G > \alpha_G$.
If some longest edge $uv$ is not incident with a shortest non-edge at both $u$ and $v$, since 
$\beta_G > \alpha_G$, Lemma \ref{alphabetalemma} gives $\alpha_G + \beta_G > 3n/4$. Similarly, if some shortest 
non-edge $uv$ is not incident with a longest edge at both $u$ and $v$, Lemma \ref{alphabetalemma} gives 
$\alpha_G + \beta_G < 3n/4$.  Clearly, these cannot hold simultaneously. This proves Lemma \ref{Satleast3}.
\end{proof}

\begin{lemma} Let $G$ be an optimal digraph with $\beta_G > \alpha_G$. For every augmenting sequence 
$S = v_1, f_1, v_2, f_2, v_3, f_3, \dots, f_k, v_{k+1}$ in $G$, 
$v_i \neq v_j$ for $i \neq j$, except possibly $v_{k+1} = v_1$.
\end{lemma}
\begin{proof}
Suppose not, and let $v_h = v_j$ with $1\leq h,j \leq k+1$, and $h$ different from $j$. 
Suppose $h,j > 1$. Then one of $f_{h-1},f_{j-1}$ is a longest edge and one is a shortest non-edge 
(since there cannot be two longest edges ending at a given vertex).
However, this contradicts that $G$ is a circular interval digraph with $\beta_G > \alpha_G$. 
Analogously, if $h,j < k+1$, one of $f_{h+1}, f_{j+1}$ is a longest edge, and the other is 
a shortest non-edge, and the same contradiction is reached. This proves that if $v_h = v_j$, then 
$\{h,j\} = \{1, k+1\}$. 
\end{proof}

\begin{thm}\label{optimalbetaalphathm}
If $G$ is an optimal digraph, then $\beta_G \leq \alpha_G$. Furthermore, 
either $\alpha_G = \beta_G$ or $\alpha_G = \beta_G + 1$.
\end{thm}
\begin{proof}
The second statement in Theorem \ref{optimalbetaalphathm}
follows from the first since $\alpha_G \leq \beta_G + 1$ by Lemma \ref{outdegbounds}, so it remains prove the first.
Suppose $G$ is an optimal digraph on $n$ vertices with $\beta = \beta_G > \alpha_G = \alpha$. 
Also let $\gamma = \gamma_G$. Let $S$ be an augmenting sequence in $G$, and
set $X = X_S$ and $Y = Y_S$. Let $S = v_1, f_1, v_2, f_2, \ldots, f_k, v_{k+1}$. 
Define $G' = (G+X)\setminus Y$, and note this is also a $2$-free circular interval 
digraph by Lemma \ref{stayCIGobs}.

Fix an extreme pair $uv \in X \cup Y$. Define $G_{uv} = G + uv$ if $uv \in Y$ and $G_{uv} = G\setminus uv$ if $uv \in X$.
For vertices $w$ different from $u,v$, define $p(w) =1$ if 
$G|\{u,v,w\}$ is a directed path (and $p(w) = 0$ otherwise), and $q(w) = 1$ if $G_{uv}|\{u,v,w\}$ is a directed path 
(and $q(w) = 0$ otherwise). Finally, let 
\begin{displaymath}
R(uv) = \sum_{w \neq u,v} q(w) - p(w).
\end{displaymath}
We now define
\begin{displaymath}
R = \sum_{uv \in X\cup Y} R(uv).
\end{displaymath}
By Lemma \ref{extremelemma}, no triple of vertices in $G$ contains three extreme pairs. 
Let $T_1$ be the number of triples of vertices $\{u,v,w\}$ such that $u,v,w$ are in clockwise order in $G$ and 
two of $uv,vw,wu$ are in $X$. Let $T_2$ be the number of $\{u,v,w\}$ such that $u,v,w$ are in clockwise
order in $G$, one of $uv,vw,wu$ is in $X$, and one is in $Y$.

Finally, for a vertex $v$, define $s^+(v) = |N^+(v)| - (\alpha - 1)$ and
$s^-(v) = |N^-(v)| - (\alpha - 1)$. Also define $t^+(v) = \beta - |N^+(v)|$ and $t^-(v) = \beta - |N^-(v)|$.
Then $s^+(v), s^-(v), t^+(v)$, and $t^-(v)$ are non-negative by Lemma \ref{outdegbounds}.

\setcounter{step}{0}
\begin{step}\label{p3d:step1}
$\tilde{P}_3(G') - \tilde{P}_3(G) = R - 2 T_1 + 2 T_2$. \end{step}

This follows from the definitions, using Lemma \ref{extremelemma} to characterize those pairs with two extreme pairs 
in $X \cup Y$.
~\bbox

\begin{step}\label{p3d:step2} For $uv\in X$, $R(uv) = \gamma + 2s^+(v) + 2s^-(u) - 2$.
\end{step}
There are $\beta - 1 + |N^+(v) \cap N^-(u)|$ vertices $w$ where $q(w) = 1$ and $p(w) = 0$. 
Since $\beta > \alpha$, Lemma \ref{alphabetalemma} implies $\alpha + \beta > 3n/4 - 1/2$, and 
combined with Lemma \ref{n/4lemma}, this gives $2\alpha + \beta \geq n$. 
By Lemma \ref{extremelemma}, this implies no vertex is non-adjacent to both $u$ and $v$. 
We can now count that there are  $n- (\beta -1) -2 - |N^+(v) \cap N^-(u)|$ vertices $w$ 
where $q(w) = 0$ but $p(w) = 1$. Recalling that $R(uv) = \sum_{w \neq u,v} q(w) - p(w)$, we have
\begin{displaymath}
R(uv) = \beta-1+|N^+(v) \cap N^-(u)| -(n-\beta-1-|N^+(v) \cap N^-(u)|), 
\end{displaymath}
or
\begin{displaymath}
R(uv) =  2\beta + 2|N^+(v) \cap N^-(u)| -n.
\end{displaymath}
We know $|N^+(v) \cap N^-(u)| = |N^+(v)| + |N^-(u)| + (\beta -1) + 2 - n$, which we can rewrite as  
$2\alpha + \beta - n - 1 + (|N^+(v)| - (\alpha-1)) + (|N^-(u)| - (\alpha -1))$. Then using the definitions of 
$s^+(v)$ and $s^-(u)$, we have
\begin{displaymath}
R(uv) =  2\beta + 2(2\alpha + \beta - n - 1) + 2s^+(v) + 2s^-(u)-n,
\end{displaymath}
which simplifies to 
\begin{displaymath}
R(uv) = 4(\alpha + \beta) -3n - 2 + 2s^+(v) + 2s^-(u).
\end{displaymath}
From the definition of $\gamma = 4(\alpha + \beta) - 3n$, this proves Step \ref{p3d:step2}.~\bbox

\begin{step}\label{p3d:step3} For $uv\in Y$, $R(uv) = 2t^+(v) + 2t^-(u) -\gamma - 2$.
\end{step}
There are $n- (\alpha +1) - |N^+(v) \cap N^-(u)|$ vertices $w$ where $q(w) = 1$ but $p(w) = 0$, 
and there are $\alpha - 1 + |N^+(v) \cap N^-(u)|$ vertices $w$ where $q(w) = 0$ and $p(w) = 1$.
Recalling that $R(uv) = \sum_{w \neq u,v} q(w) - p(w)$, we have 
\begin{displaymath}
R(uv) = n-\alpha -1 - |N^+(v) \cap N^-(u)| - (\alpha - 1 + |N^+(v) \cap N^-(u)|),
\end{displaymath}
or 
\begin{displaymath}
R(uv) = n - 2\alpha - 2|N^+(v) \cap N^-(u)|.
\end{displaymath}
We know $|N^+(v) \cap N^-(u)| = |N^+(v)| + |N^-(u)| + (\alpha -1) + 2 - n$, which we can rewrite as  
$\alpha + 2\beta - n + 1 - (\beta - |N^+(v)|) -  (\beta - |N^-(u)|)$. Then 
using the definitions of $t^+(v), t^-(u)$, we have
\begin{displaymath}
R(uv) = n - 2\alpha - 2(\alpha + 2\beta -n + 1) + 2t^+(v) + 2t^-(u), 
\end{displaymath}
which simplifies to 
\begin{displaymath}
R(uv) = 3n - 4(\alpha + \beta) + 2t^+(v) + 2t^-(u) - 2.
\end{displaymath}
From the definition of $\gamma = 4(\alpha + \beta) - 3n$, this proves Step \ref{p3d:step3}.~\bbox

\begin{step}\label{p3d:step4} $\tilde{P}_3(G') - \tilde{P}_3(G) \ge 0$. \end{step}

For $1 \leq i \leq k$, if $v_iv_{i+1} \in X$, then $v_i$ has out-degree $\beta$ and 
$t^+(v_i) = 0$; similarly, if $v_iv_{i+1} \in Y$, then $s^+(v_i) = 0$. For $2 \leq i \leq k+1$, 
if $v_{i-1}v_i \in X$, then $v_i$ has in-degree $\beta$ and $t^-(v_i) = 0$; analogously, 
if $v_{i-1}v_i \in Y$, then $s^-(v_i) = 0$. For $2 \leq i \leq k$, we note that the definition 
of an augmenting sequence implies that one of $v_{i-1}v_i, v_iv_{i+1}$ is in $X$ and the other in $Y$.
Then by Steps \ref{p3d:step2} and \ref{p3d:step3}, for $3 \leq i \leq k$
\begin{equation}\label{Rii}
R(v_{i-1}v_{i}) = 
\begin{cases}
\gamma - 2 & \text{if }v_{i-1}v_i \in X \\
-\gamma - 2 & \text{if }v_{i-1}v_i \in Y.
\end{cases}
\end{equation}

We see $T_1 = 0$ unless $v_1 = v_{k+1}$, $k$ is odd, and $v_1v_2, v_kv_{k+1}$ are 
both in $X$, and in that case $T_1 = 1$. Also $T_2 = k-1$ unless $v_1 = v_{k+1}$ and $k$ is even, 
and in that case $T_2 = k$.

First, suppose $v_1 = v_{k+1}$ and $k$ is odd. Then $v_1v_2$ and $v_kv_1$ must be in $X$ by Lemma \ref{extremelemma}. 
By Step \ref{p3d:step1}, $\tilde{P}_3(G') - \tilde{P}_3(G)  =  R - 2 T_1 + 2 T_2$, which in conjunction with
equation (\ref{Rii}) and the earlier argument giving $s^+(v_2) = s^-(v_{k}) = 0$ implies
\begin{equation*}
\tilde{P}_3(G') - \tilde{P}_3(G) =  (\gamma + 2s^-(v_1) - 2) + (\gamma + 2s^+(v_{k+1}) -2) - 2(k-2) - \gamma -2T_1 + 2T_2 .
\end{equation*}
Using that $v_1 = v_{k+1}$ and substituting $T_1 = 1$ and $T_2 = k-1$, we have
\begin{equation*}
\tilde{P}_3(G') - \tilde{P}_3(G) =  \gamma + 2(s^-(v_1) + s^+(v_1) - k - 1 + (k-1)) = \gamma + 2s^-(v_1) + 2s^+(v_1) - 4.
\end{equation*}
We know $|N^-(v_1)| = |N^+(v_1)| = \beta$, so $s^-(v_1) = s^+(v_1) = \beta - \alpha + 1$. Since $\beta > \alpha$, 
$s^-(v_1), s^+(v_1) \geq 2$. Substituting in the above inequality, we have
\begin{equation*}
\tilde{P}_3(G') - \tilde{P}_3(G) \geq  \gamma + 4. 
\end{equation*}
Since $\gamma \geq -1$ by Lemma \ref{gammalemma}, $\tilde{P}_3(G') > \tilde{P}_3(G)$, as required. 

Now suppose that $v_1 = v_{k+1}$ and $k$ is even.
Since $v_1 = v_{k+1}$, equation (\ref{Rii}) holds for $2 \leq i \leq k+1$. Also,
by Step \ref{p3d:step1},
\begin{equation}\label{keveneq}
\tilde{P}_3(G') - \tilde{P}_3(G) =  R - 2 T_1 + 2 T_2 =  - 2k -2T_1 + 2T_2.
\end{equation}
Substitution of $T_1 = 0$, $T_2 = k$ in equation (\ref{keveneq})
yields $\tilde{P}_3(G') - \tilde{P}_3(G) = 0$, as required. 

Thus we may assume $v_1 \neq v_{k+1}$. We note Lemma \ref{Satleast3} implies that if $v_1 \neq v_{k+1}$, then
$v_1v_2, v_kv_{k+1}$ are either both in $X$ or both in $Y$, and $k$ is odd.
Let $\mu = -1$ if $v_1v_2 \in X$ and $\mu = 1$ otherwise.

By Step \ref{p3d:step1}, $\tilde{P}_3(G') - \tilde{P}_3(G)  =  R - 2 T_1 + 2 T_2$, which in conjunction with
equation (\ref{Rii}) implies
\begin{equation}\label{koddeq}
\tilde{P}_3(G') - \tilde{P}_3(G) =  R(v_1v_2) + R(v_kv_{k+1}) - 2(k-2) + \mu \gamma -2T_1 + 2T_2,
\end{equation}
where
\begin{displaymath}
R(v_1v_2) = 
\begin{cases}
\gamma + 2s^-(v_1) - 2 & \text{if }v_1v_2 \in X \\
2t^-(v_1) -\gamma - 2 & \text{if }v_1v_2 \in Y,
\end{cases}
\end{displaymath}
and
\begin{displaymath}
R(v_kv_{k+1}) = 
\begin{cases}
\gamma + 2s^+(v_{k+1}) - 2 & \text{if }v_kv_{k+1} \in X \\
2t^+(v_{k+1}) -\gamma - 2 & \text{if }v_{k}v_{k+1} \in Y.
\end{cases}
\end{displaymath}
Then equation (\ref{koddeq}) can be simplified to
 \begin{equation*}
\tilde{P}_3(G') - \tilde{P}_3(G) =  
\begin{cases}
\gamma + 2s^-(v_1) + 2s^+(v_{k+1}) - 2k -2T_1 + 2T_2 
& \text{if }v_1v_2 \in X\\
- \gamma + 2t^-(v_1) + 2t^+(v_{k+1}) - 2k -2T_1 + 2T_2 
& \text{if }v_1v_2 \in Y.
\end{cases}
\end{equation*}
Recalling that $T_1 = 0$ and $T_2 = k-1$, we have 
\begin{equation*}
\tilde{P}_3(G') - \tilde{P}_3(G) =  
\begin{cases}
\gamma + 2s^-(v_1) + 2s^+(v_{k+1}) - 2
& \text{if }v_1v_2 \in X\\
- \gamma + 2t^-(v_1) + 2t^+(v_{k+1}) - 2
& \text{if }v_1v_2 \in Y.
\end{cases}
\end{equation*}

First, suppose $v_1v_2 \in X$. Then $s^-(v_1), s^+(v_{k+1}) \geq 1$, 
since otherwise $S$ is not maximal. Since $\gamma \geq -1$ by Lemma \ref{gammalemma},
this proves $\tilde{P}_3(G') - \tilde{P}_3(G) > 0$. On the other hand, suppose
$v_1v_2 \in Y$. Then $t^-(v_1), t^+(v_{k+1}) \geq 1$ by the maximality of $S$. 
Now $\gamma \leq 1$ by Lemma \ref{gammalemma}, and again $\tilde{P}_3(G') - \tilde{P}_3(G) > 0$.
This completes the proof of Step \ref{p3d:step4}.~\bbox 

We observe that $\xi(G') < \xi(G)$ follows immediately from the definition of $\xi(G)$ and 
the fact $\beta < \alpha$. Yet we have now contradicted the optimality of $G$. This proves Theorem \ref{optimalbetaalphathm}.
\end{proof}

Finally, we prove a lemma relating the number of induced $3$-vertex paths in a general circular interval 
digraph with longest edge of length $\beta$ to the number in $G_\beta$. We need two further definitions. 

Let $H_{\beta}$ be the subgraph of $G_{\beta}$ with the same vertex set, and 
$E(H_\beta) = \{uv\,:\, d(u,v) = \beta\}$. Also, 
for $X \subseteq E(H_\beta)$, let $t(X)$ be the number of vertices of $H_\beta$ which 
are incident with exactly one edge in $X$. 

\begin{lemma}\label{gbetalemma1}
Let $n\ge 4$, and let $\beta$ be an integer satisfying $-2 \le 8\beta -3n \le 2$. 
Then for all $X \subseteq E(H_\beta)$, 
\begin{displaymath}
|X|(8\beta - 3n) + t(X) + n(n-2\beta -1)(2\beta -n/2+1)\le n^3/16.
\end{displaymath}

\end{lemma}

\begin{proof}
Let $\delta = 8\beta - 3n$. Then $-2 \le \delta \le 2$, and (eliminating 
$\beta$) we must show that
\begin{displaymath}
|X|\delta + t(X) + n(n-\delta -4)(n + \delta +4)/16 \le n^3/16,
\end{displaymath}
that is, 
\begin{equation}\label{deltaineq}
|X|\delta + t(X) \le n(\delta +4)^2/16
\end{equation}
for all $X\subseteq E(H_\beta)$.

Let $t = t(X)$, and $Y = E(H_\beta)\setminus X$. In $G_\beta$, every vertex is incident with 
two edges of length $\beta$. Since $X \cup Y = E(H_\beta)$, and $t$ counts vertices 
which are incident with exactly one edge in $X$, we have that $2|Y|\geq t$, 
$2|X|\geq t$, and $|X| + |Y| = n$.

\setcounter{case}{0}
\begin{case} $\delta = 0$.\end{case}

Since $\delta = 0$, equation (\ref{deltaineq}) becomes $t \leq n$, 
which is clear since $G$ has $n$ vertices.

\begin{case} $\delta = 1$.\end{case}

Substituting into inequality (\ref{deltaineq}), we must show that $|X| + t \leq 25 n/16.$
Since $2|Y|\geq t$ and $2|X|\geq t$, it follows that $6|Y| + 2|X|\geq 4t$.
Using $|X|+|Y| = n$ to eliminate $|Y|$ gives
$6(n-|X|) + 2|X| \geq 4t$, that is, $|X| + t \leq 3n/2 < 25n/16$, as required.

\begin{case} $\delta = 2$.\end{case}

In this case, equation (\ref{deltaineq}) becomes $2|X| + t \leq 9 n/4.$
But since $2|Y|\geq t$ and $|Y| = n-|X|$, we have $2(n -|X|) \geq t$, or  
$2|X| + t \leq 2t \leq 2n < 9n/4$, as required.

\begin{case} $\delta = -1$.\end{case}

When $\delta = -1$, we need to show $t -|X| \leq 9n/16$ to prove the inequality in (\ref{deltaineq}).
If $|X|\leq n/2$ then $t \leq 2|X| \leq |X| + n/2$. 
If $|X|> n/2$, then $t \leq 2|Y| = 2(n-|X|) \leq n/2 + |X|$. In both cases, 
$t \leq |X| + n/2 < |X| + 9n/16$, as required.

\begin{case} $\delta = -2$.\end{case}

Finally, when $\delta = -2$, proving (\ref{deltaineq}) requires $t - 2|X| \leq n/4$.
But $2|X| \geq t$, so this is trivial. This proves Lemma \ref{gbetalemma1}.
\end{proof}

\begin{lemma}\label{gbetalemma2} Let $G =  G_\beta \setminus X$, where $X \subseteq E(H_\beta)$, and $8\beta-3n \geq 2$.
Then $\tilde{P}_3(G) = \tilde{P}_3(G_\beta) + |X|(8\beta - 3n) + t(X)$. 
\end{lemma}
\begin{proof}
For each edge $uv$ in $X$, the number of induced $3$-vertex paths using both of $u,v$ which are in 
$G$ and not $G_\beta$ is $\beta - 1 + (3\beta - n-1)$, plus one for each vertex $w$ so that $uw$ or $wv$ 
is in $X$. The number of induced $3$-vertex paths using $u$ and $v$ which are in $G_\beta$ and not $G$ is 
$2(n-2\beta -1)$. Summing over all $uv$ in $X$, we see that 
$\tilde{P}_3(G) = \tilde{P}_3(G_\beta) + |X|(8\beta - 3n) + t(X)$, by definition of $t(X)$.
This proves Lemma \ref{gbetalemma2}.
\end{proof}

\noindent{\bf Proof of Theorem \ref{CIGthm}.}
Let $G$ be a digraph on $n$ vertices.
If $n \leq 2$, then $\tilde{P}_3(G) = 0 \leq n^3/16$, and if $n = 3$, then $\tilde{P}_3(G) \leq 1 \leq 27/16$. 
So we may assume $n \geq 4$, and that $G$ is optimal.  It follows from Theorem \ref{optimalbetaalphathm} that
every optimal digraph $G$ with maximum edge length $\beta$ 
can be written as $G_\beta \setminus X$ for some set $X \subseteq H_\beta$.
We now show that every choice of $X$ gives $\tilde{P}_3(G) \leq n^3/16$. Let $\alpha = \alpha_G$ and $\beta = \beta_G$. 
By Lemma \ref{optimalbetaalphathm}, either $\alpha = \beta$, or $\alpha = \beta + 1$.

Suppose $\alpha = \beta + 1$. Then $X = \emptyset$, and $G = G_\beta$. A straightforward calculation gives that 
\begin{equation}\label{p3gbetaeqn}
\tilde{P}_3(G_\beta) = n(n-2\beta-1)(2\beta - n/2 + 1). 
\end{equation}
Let $x = 2\beta + 1$. 
Then we need to show $n(n-x)(x-n/2) \leq n^3/16$, or $x(3n/2 - x) \leq 9n^2/16$.
Now, Lemma \ref{alphabetalemma} implies that $3n/4 - 1/2 \leq x \leq 3n/4 + 1/2$. We see that 
$x(3n/2 - x)$ is maximized when $x = 3n/4$, where it is equal to $9n^2/16$. This proves that when 
$\alpha = \beta + 1$, $\tilde{P}_3(G) \leq n^3/16$. 

Thus we may assume $\alpha = \beta$. Lemma \ref{alphabetalemma} now gives $3n/8 - 1/4 \leq \beta \leq 3n/8 + 1/4$, or 
$3n - 2 \leq 8\beta \leq 3n + 2$. Theorem \ref{CIGthm} then follows directly from equation (\ref{p3gbetaeqn}), together with 
Lemmas \ref{gbetalemma1} and \ref{gbetalemma2}. ~\bbox

\section{Four-Vertex Paths in $3$-free Digraphs}\label{p4section}

The main result of this section is: 
\begin{thm}\label{P_4thm}
If $G$ is a $3$-free digraph on $n$ vertices, then
$P_4(G) \leq \frac{4}{75}n^4.$ 
\end{thm}

We first establish some notation and two key lemmas. 

Let $G$ be a $3$-free digraph on $n$ vertices. We will use the term {\em square} to 
refer to a subgraph of $G$ which is a directed cycle of length four. If $X \subseteq V(G)$ with $|X| = 4$, let
$t(X)$ be the number of $4$-vertex directed paths with vertex set $X$. We observe that since 
$G$ is $3$-free, $t(X) \in \{0,1,4\}$ for every such $X$. This
motivates the following definitions. Let $R$ be the number of four-tuples 
of distinct vertices $(a,b,c,d)$ such that $t(\{a,b,c,d\}) = 1$. Let $S$ be the number of 
four-tuples of distinct vertices $(a,b,c,d)$ such that 
$G|\{a,b,c,d\}$ is a square (equivalently, $t(\{a,b,c,d\}) = 4$).
Then $S$ is $24$ times the number of squares.
Define $N$ to be the set of four-tuples of vertices not counted by either $R$ or $S$, 
so $|N| = n^4 - R- S$. For distinct vertices $u,v$, let
$M(u,v)$ be the set of all vertices $x$ such that $(u,x,v)$ is an induced $3$-vertex path. Set
$m(u,v) = |M(u,v)|$, the number of induced directed $3$-vertex paths starting at $u$ and ending 
at $v$. Finally, define $T = \tilde{P}_3(G)$.

\begin{lemma}\label{upperSlemma} 
In a $3$-free digraph $G$, $S \leq \frac{3n}{2}T$.
\end{lemma}
\begin{proof}
We will write $P \sqsubset G$ to mean $P$ is a (directed) path of $G$, and $\sum_P, \sum_\Gamma$ 
to mean the sum over all induced $3$-vertex paths in $G$ and the sum over all squares in $G$, respectively.
For each square $\Gamma = a\d b\d c\d d\d a$ in $G$, define 
\begin{equation*}
\omega(\Gamma) = \frac{1}{m(c,a)} + \frac{1}{m(d,b)} + \frac{1}{m(a,c)} + \frac{1}{m(b,d)}. 
\end{equation*}
Now, since $m(a,c) + m(c,a) + m(b,d) + m(d,b) \leq n$ (each path has a middle vertex, and 
no vertex can serve as the middle of two of the paths counted since $G$ has no 
directed cycle of length at most three), $\omega(\Gamma) \geq 16/n$ for all $\Gamma$. Since 
there are $S/24$ squares, it follows that 
\begin{equation*}
\sum_\Gamma \omega(\Gamma) \geq \frac{16}{n}\left(\frac{S}{24}\right) = \frac{2}{3n}S. 
\end{equation*}
For an induced $3$-vertex path $P = u\d w\d v$ in $G$, let 
\begin{equation*}
\omega(P) = \frac{1}{m(v,u)} \left|\{\text{squares }\Gamma: P \sqsubset \Gamma\}\right|.
\end{equation*}
We claim that $\omega(P) = 1$ for all $P$. The squares containing $P$ are 
of the form $u\d w\d v\d x\d u$ where $(v,x,u)$ is also an induced $3$-vertex path. Since $G$ is $3$-free, 
every $4$-cycle is induced, so every choice of $x \in M(v,u)$ gives a square, proving
$\omega(P) = m(v,u)\cdot \frac{1}{m(v,u)} = 1$. 
Then $\sum_P \omega(P) = \sum_P 1 = T$ by definition.  \\

Finally, we show $\sum_P \omega(P) = \sum_\Gamma \omega(\Gamma)$. Below, let $P$ be $u\d w\d v$. 
Then 
\begin{equation*}
\begin{split}
\sum_P \omega(P)  &=  \sum_P \frac{1}{m(v,u)} \left|\{\text{squares }\Gamma: P \sqsubset \Gamma\}\right| 
\\ &=  \sum_P \sum_{\Gamma \sqsupset P} \frac{1}{m(u,v)} 
 =  \sum_\Gamma \sum_{P \sqsubset \Gamma} \frac{1}{m(u,v)}
 =  \sum_\Gamma \omega(\Gamma).
\end{split}
\end{equation*}
We now have $T = \sum_P \omega(P) = \sum_\Gamma \omega(\Gamma) \geq \frac{2}{3n}S$, or 
$S \leq \frac{3n}{2}T.$ 
This proves Lemma \ref{upperSlemma}.
\end{proof}

\begin{lemma}\label{lowerNlemma} 
If $G$ is a $3$-free digraph, then $|N| \geq \frac{2}{3}S$.
\end{lemma}
\begin{proof}
Let $\Gamma = a\d b\d c\d d\d a$ be a square in $G$. Define 
\begin{equation*}
\omega(\Gamma) = 2\left(\frac{(m(b,d) + m(d,b))^2}{m(a,c)m(c,a)} 
+ \frac{(m(a,c) + m(c,a))^2}{m(b,d)m(d,b)}\right).
\end{equation*}
Again, $m(a,c) + m(c,a) + m(d,b) + m(b,d) \leq n$, and by Cauchy-Schwarz,  
$\omega(\Gamma) \geq 16$ (since we know that $m(u,v) > 0$ for each 
relevant $u,v$). Since there are $S/24$ squares, we have
\begin{equation}\label{gammabound}
\frac{2}{3}S \leq \sum_\Gamma \omega(\Gamma).
\end{equation}

Given a four-tuple of vertices $\pi = (p,q,r,s)$ and a square $\Gamma$, 
we say they are {\it associated}, and write $\pi \sim \Gamma$, 
if there exist vertices $u,v$ such that 
$\Gamma = p \d u \d q \d v \d p$ and $r,s \in M(u,v)\cup M(v,u)$. 
Note that for a square $\Gamma = a\d b\d c\d d\d a$, 
the four-tuples associated with it are precisely those of the forms 
$(a,c,x,y) \text{ or } (c,a,x,y)$ where $x,y \in M(d,b)\cup M(b,d)$, 
and $(b,d,x,y) \text{ or } (d,b,x,y)$ with $x,y \in M(a,c)\cup M(c,a)$. 

Now, for a four-tuple of vertices $\pi = (p,q,r,s)$, define $\omega(\pi)$ as follows:
\begin{equation*}
\omega(\pi) = \frac{|\{\Gamma: \Gamma  \sim \pi\}|}{m(p,q)m(q,p)}.
\end{equation*}
Note that $\omega(\pi) \leq 1$, since the number of squares associated with
$\pi$ is at most $m(p,q)m(q,p)$ by definition. Then 
\begin{equation}\label{pibound}
\sum_{\pi \in N} \omega(\pi) \leq \sum_{\pi \in N} 1 \leq |N|.
\end{equation}

Next, if $\Gamma = a\d b\d c\d d\d a$ is a square in $G$, 
we show that $\pi \sim \Gamma$ implies $\pi \in N$. Without
loss of generality, we may let $\pi = (a,c,x,y)$. 
We need to show that there is no 4-vertex path with vertex set $\{a,c,x,y\}$.
This is clear if $a,c,x,y$ are not all distinct, so we assume they are 
distinct. Since $b$ is adjacent to every vertex in $M(b,d)$ and from every 
vertex in $M(d,b)$, there is no edge from $M(b,d)$ to $M(d,b)$, since 
otherwise there would be a directed triangle. Similarly,
there is no edge from a vertex in $M(d,b)$ to a vertex in $M(b,d)$. Consequently, 
if $X$ is a set of four vertices so that $G|X$ has a $4$-vertex path 
as a subgraph and $X \subseteq M(b,d) \cup M(b,d)$, then $X \subseteq M(b,d)$ or 
$X \subseteq M(b,d)$. So not both of $a,c$ are in $X$. This proves that 
every $\pi$ associated with $\Gamma$ belongs to $N$. \\

This observation allows us to relate $\sum_\Gamma \omega(\Gamma)$ to 
$\sum_{\pi \in N} \omega(\pi)$. Assuming $\pi = (p,q,r,s)$ for
the purposes of writing $\omega(\pi)$,  
\begin{displaymath}
\sum_\Gamma \omega(\Gamma)= 
\sum_\Gamma{\sum_{\pi \sim \Gamma} \frac{1}{m(p,q)m(q,p)} } = 
\sum_{\pi \in N} \sum_{\Gamma \sim \pi} \frac{1}{m(p,q)m(q,p)} =   
\sum_{\pi \in N} \omega(\pi).
\end{displaymath}

Combining this with (\ref{gammabound}) and (\ref{pibound}), we have
\begin{displaymath}
\frac{2S}{3} \leq \sum_\Gamma \omega(\Gamma) = \sum_{\pi \in N} \omega(\pi) \leq |N|.
\end{displaymath}
This proves Lemma \ref{lowerNlemma}.
\end{proof}

{\bf Proof of Theorem \ref{P_4thm}:}
Note that $n^4 = R + S + |N|$ by definition. We can also express the number
of $4$-vertex paths $P_4(G)$ in terms of these parameters, as $24P_4(G) = 4S + R$. Combining these
equalities, we write 
\begin{equation}\label{PSNequality}
24P_4(G) = n^4 + 3S- |N|.
\end{equation}
To prove an upper bound for $P_4(G)$, it then suffices to bound $S$ from above and $|N|$ from below. 
From Lemmas \ref{upperSlemma} and \ref{lowerNlemma}, we have $S \leq \frac{3n}{2}T$ and $|N| \geq \frac{2}{3}S$.
Combining these with (\ref{PSNequality}), we see that:  
\begin{equation*}
24P_4(G) \leq  n^4 + \frac{7}{3}S \leq n^4 + \frac{7}{2}nT. 
\end{equation*}
But $T \leq \frac{2}{25}n^3$ by Theorem \ref{bondythm},  and so
$24P_4(G) \leq (1 + 7/25)n^4$, or $P_4(G) \leq \frac{4}{75}n^4$, as desired. ~\bbox

It follows immediately from Theorem \ref{P_4thm} that every $3$-free digraph on $n$ vertices has a vertex of 
out-degree at most $\sqrt[3]{4/75}n \approx .3764n$. Note that if Conjecture \ref{thomasseconj} holds, we could 
replace Bondy's bound on $P_3$ by $n^3/15$, and the proof of Theorem \ref{P_4thm} would then give
$P_4(G) < \frac{1}{19.45}n^4 \approx .0514n^4$, implying the existence of a vertex with
out-degree at most $\sqrt[3]{\frac{1}{19.45}}n \approx .37184n$.

\end{document}